\documentclass[oneside,english]{amsart}
\usepackage[T1]{fontenc}
\usepackage[latin9]{inputenc}
\usepackage{geometry}
\usepackage{verbatim}
\usepackage{amstext}
\usepackage{amsthm}
\usepackage{bbm}
\usepackage{hyperref}
\usepackage{graphicx}

\hypersetup{
   colorlinks = true,
   urlcolor = blue,
   linkcolor = red
}

\makeatletter
\numberwithin{equation}{section}
\numberwithin{figure}{section}
\theoremstyle{plain}
\newtheorem{thm}{\protect\theoremname}
  \theoremstyle{plain}
  
  \theoremstyle{remark}
  
  \theoremstyle{plain}
  \newtheorem{cor}[thm]{\protect\corollaryname}

\makeatother

\usepackage{babel}
  \providecommand{\corollaryname}{Corollary}
  \providecommand{\lemmaname}{Lemma}
  \providecommand{\remarkname}{Remark}
\providecommand{\theoremname}{Theorem}
\newcommand{\E}{\mathbb{E}}
\newcommand{\A}{\mathcal{A}}

\begin{document}

\title{Refinements of the Bell and Stirling numbers}

\author{Tanay Wakhare$^\ast$}
\address{$^\ast$~University of Maryland, College Park, MD 20742, USA}
\email{twakhare@gmail.com}

\begin{abstract}We introduce new refinements of the Bell, factorial, and unsigned Stirling numbers of the first and second kind that unite the derangement, involution, associated factorial, associated Bell, incomplete Stirling, restricted factorial, restricted Bell, and $r$-derangement numbers (and probably more!). By combining methods from analytic combinatorics, umbral calculus, and probability theory, we derive several recurrence relations and closed form expressions for these numbers. By specializing our results to the classical case, we recover explicit formulae for the Bell and Stirling numbers as sums over compositions. 
\end{abstract}

\maketitle

\section{Introduction}
The Bell and Stirling numbers have been studied for over a century because of their importance to many combinatorial problems. They frequently arise in enumeration problems in combinatorics, and also satisfy many complex identities and inter-relations \cite{Comtet,Mansour}. In this paper, we present refinements of the Bell and Stirling numbers that preserve many essential structural properties of their classical counterparts. In particular, they satisfy many of the same inter-relations and recurrence relations because the associated generating functions have a simple form which enables us to follow classical proofs of many identities. They also unite several previously disparate generalizations of the Stirling and Bell numbers. 

We use a variety of methods to then explore their properties; we begin with an analytic combinatoric construction of their generating function, followed by umbral and generating function methods to derive identities for these numbers. Finally, an explicit probabilistic representation allows us to construct some new and extremely interesting expressions for the Bell and Stirling numbers as sums over compositions.

Let $S$ be a nonempty (and possibly infinite) set of indices. We consider \textbf{generalized Bell numbers} $B_{n,S}$, defined as the number of ways to partition $n$ into blocks, where each block has a size which is found in $S$. We also consider the \textbf{generalized factorial numbers} $A_{n,S}$, which count the number of permutations of $n$ labeled elements into cycles, such that each cycle has a number of elements that belongs to the index set $S$. For example, with $S=\{1,3\}$, $A_{4,S} = 9$ and $B_{4,S} = 5$.  The relevant partitions for $B_{4,S}$ are 
\begin{itemize}
\item
\{\{1\},\{2\},\{3\},\{4\}\}, 
\item
\{\{1,2,3\},\{4\}\}, 
\item
\{\{1,2,4\},\{3\}\},
\item
\{\{1,3,4\},\{2\}\},
\item
 \{\{2,3,4\},\{1\}\},
\end{itemize}
and the relevant permutations for $A_{4,S}$ are
\begin{itemize}
\item
(1)(2)(3)(4)
\item
(1)(342)
\item
(1)(432)
\item
(2)(341)
\item
(2)(431)
\item
(3)(241)
\item
(3)(421)
\item
(4)(231)
\item
(4)(321)
\end{itemize}

We introduce $$\begin{bmatrix}n \\k \end{bmatrix}_S,$$ the \textbf{generalized (unsigned) Stirling numbers of the first kind}, which count the number of permutations of $n$ elements with $k$ cycles, such that each cycle has a cardinality found in $S$. We also introduce $$\begin{Bmatrix}n \\k \end{Bmatrix}_S,$$ the \textbf{generalized Stirling numbers of the second kind}, which count the number of partitions of $n$ into $k$ boxes, such that each box has a cardinality found in $S$. These have two free variables, and contain more combinatorial informations than the Bell and factorial numbers.

We also need to define initial values, since setting the following values allows us to concisely state consistent generating function identities. These definitions are consistent with initial values for the classical case \cite{Comtet}. For $n\geq1$, we set $$\begin{bmatrix}n \\0 \end{bmatrix}_S=\begin{bmatrix}0 \\k \end{bmatrix}_S=\begin{bmatrix}n \\0 \end{bmatrix}_S=\begin{Bmatrix}0 \\k \end{Bmatrix}_S=0. $$ We also set $$ \begin{bmatrix}0 \\0 \end{bmatrix}_S=\begin{bmatrix}0 \\0 \end{bmatrix}_S=1,$$ and note that from the combinatorial definition $$ \begin{bmatrix}n \\k \end{bmatrix}_S=\begin{bmatrix}n \\k \end{bmatrix}_S=0$$ for $k>n$.

Directly from the combinatorial definitions, we see that we have the equations $$B_{n,S} = \sum_{k=0}^n \begin{Bmatrix}n \\k \end{Bmatrix}_S $$ and $$A_{n,S} = \sum_{k=0}^n \begin{bmatrix}n \\k \end{bmatrix}_S. $$ Therefore, we set the initial values  $$A_{0,S}=B_{0,S}=1. $$ %Because of their natural correspondence with the Stirling numbers, $B_{n,S}$ and $A_{n,S}$ can be viewed as natural duals, just like $\begin{bmatrix}n \\k \end{bmatrix}_S$ and $\begin{Bmatrix}n \\k \end{Bmatrix}_S$.

These numbers have been studied for many special values of the index set $S$. The case $S=\mathbb{Z}_{\geq 1}=\{1,2,\ldots\}$ yields the classical Bell, Stirling, and factorial numbers since these do not have any restrictions on cycle length or box size. The cases $S=\{1,2,\ldots,k\}$ and $S=\{k,k+1,\ldots\} $ are referred to as $B_{n,\leq k}$, $B_{n,\geq k}$, $A_{n,\leq k}$, and  $A_{n,\geq k}$, and have been extensively studied several times \cite{Miksa, Moll1}. In another direction, $S=\{1,2\}$ means that $B_{n,S}=A_{n,S}$ yield the involution numbers \cite{Tewodros} and $S=\{2,3,4,\ldots\}$ leads to the derangement numbers \cite{Bona}. Extending our study of Stirling numbers from these particular index sets $S$ to the general case allows us to prove general theorems that apply to every special case. It also leads to identities that are very non-intuitive combinatorially, guiding future work.

%TODO rewrite last para

\begin{table}[]
\centering
\caption{Specializations}
\label{my-label}
\scalebox{1.0}{

\begin{tabular}{|l|l|}
\hline
\textbf{Set}                                     & \textbf{Classical object} \\ \hline
$\{1,2,3,\ldots\}$ & Bell and Stirling, $n!$                                    \\ 
$ \{k,k+1,\ldots\}$                     & associated Bell, associated Stirling, associated factorial ($r$-derangement) \\ 
$ \{1,2,3,\ldots ,k\}$                   & restricted Bell, restricted Stirling, restricted factorial \\ 
$\{1,2\}$                               & involution numbers                                         \\ 
$ \{2,3,4,\ldots\}$                     & derangement numbers                                        \\ \hline
\end{tabular}

}
\end{table}

\section{Symbolic Methods}
Here we present the basic ideas of \textit{analytic combinatorics}, which trivializes the proof of our generating functions. For further details of the theory described in this section, refer to Flajolet and Sedgewick \cite{Flajolet}. The main idea of this symbolic method is to write our combinatorial constructions as the composition of several basic operations on a single element. We can then algorithmically read off a generating function for our combinatorial quantity from this sequence of operations.

We consider exponential generating functions (EGFs), which naturally correspond to labeled objects. Ordinary generating functions correspond to unlabeled objects. Some of the most basic constructions \textit{for labeled elements} (there are subtle differences for unlabeled objects) are $SUM$, $PROD$, $SEQ$, $SET$, and $CYC$. Given that $A(z)=\sum_{n}a_n \frac{z^n}{n!}$ and $B(z) = \sum_n b_n \frac{z^n}{n!}$ are the EGFs of $\{a_n\}$ and $\{b_n\}$, the operation $SUM$ produces an EGF for $\{a_n+b_n\}$ - the number of elements of size $n$, following pointwise addition of $A$ and $B$. This is trivially $A(z)+B(z) = \sum_n (a_n+b_n)\frac{z^n}{n!},$ but other operations allow us to construct very nontrivial generating functions.

%%typeset caps like flajolet
The remaining operations are:
\begin{itemize}
\item
$PROD(\mathcal{B},\mathcal{C})$, which gives an EGF for the cardinalities of the ``labeled product" of $\{a_n\}$ and $\{b_n\}$, which consists of the set of ordered pairs $\{(a_n,b_n)\}$ after an order consistent relabeling. This is given by $A(z) \times B(z)$
\item
$SEQ(\mathcal{B})$, which corresponds to an EGF for the number of sequences of a given size with parts in $\mathcal{B}$, is given by $\frac{1}{1-B(z)}$. $SEQ_k(\mathcal{B})$, corresponding to the size of all $k$-element sequences, is given by ${B(z)^k}$.
\item
$SET(\mathcal{B})$ corresponds to forming all sequences, taken modulo an equivalence relation identifying all sequences that are permutations of each other. This is given by $\exp\left(B(z)\right)$. $SET_k(\mathcal{B})$, the set of all $k$-sequences modulo this same equivalence relation, is given by $\frac{B(z)^k}{k!}.$
\item
$CYC(\mathcal{B})$ corresponds to $SEQ(\mathcal{B})$, taken modulo an equivalence relation identifying sequences whose elements are \textit{cyclic} permutations of each other, and is given by $\log \frac{1}{1-B(z)}$. $CYC_k(\mathcal{B})$, the set of $k$-sequences modulo this same equivalence relation, is given by $\frac{B(z)^k}{k}$. 
\end{itemize}

Using just these few operations (there are more!) allows us to write
\begin{thm}\label{main}
We have the following constructions:
$$\sum_{n=0}^\infty \begin{bmatrix}n \\k \end{bmatrix}_S \frac{z^n}{n!}  = SET_{k} \left( \sum_{s \in S}CYC_{s}(\mathcal{Z}) \right)=  \frac{1}{k!} \left(\sum_{s \in S} \frac{z^{s}}{s} \right)^k,$$
$$\sum_{n=0}^\infty \begin{Bmatrix}n \\k \end{Bmatrix}_S \frac{z^n}{n!}  = SET_{k} \left( \sum_{s \in S}SET_{s}(\mathcal{Z}) \right)= \frac{1}{k!} \left(\sum_{s \in S} \frac{z^{s}}{s!} \right)^k,$$
$$\sum_{n=0}^\infty A_{n,S}\frac{z^n}{n!} = SET\left( \sum_{s\in S} CYC_{s}(\mathcal{Z})  \right) =\exp \left(\sum_{s \in S} \frac{z^{s}}{s} \right),$$
$$\sum_{n=0}^\infty B_{n,S}\frac{z^n}{n!} = SET\left( \sum_{s\in S} SET_{s}(\mathcal{Z})  \right)=\exp \left(\sum_{s \in S} \frac{z^{s}}{s!} \right),$$
where $\mathcal{Z}$ is an ``atomic class" of a single element of size one.
\end{thm}
\begin{proof}We give an example for how to translate this into a generating function; the given EGFs follow from interpreting $\mathcal{Z}$ as the variable $z$, and then formulaically applying the rules for symbolic constructions.

We begin with the combinatorial definition of $\begin{bmatrix}n \\k \end{bmatrix}_S$, the number of permutations of $n$ labeled elements with $k$ cycles, such that each cycle has a cardinality found in $S$. We first form the sum $ \sum_{s \in S}CYC_{s}(\mathcal{Z})$. Since each $CYC_{s}$ operator is invariant under a cyclic permutation of its sequences, it enumerates the number of cycles of length $s$. The sum over $s \in S$ then corresponds to the fact that we allow a cycle to have any cardinality found in our index $S$. $SET_k$ then composes $k$ cycles, while ignoring the order in which the cycles are permuted. The coefficient of $z^n$ in the corresponding EGF will therefore count the number of ways $n$ elements can be decomposed as permutations with cycle lengths in $S$ -- which is precisely $\begin{bmatrix}n \\k \end{bmatrix}_S$. 

%We give an example for how to translate this into a generating function; the other given EGFs follow from interpreting $\mathcal{Z}$ as the variable $z$, and then formulaically applying the rules for set constructions. We begin by setting $ SET_k(\mathcal{Z}) = \frac{z^k}{k!} $ and then composing this with $ SET(\mathcal{B}) = \exp \left( B(z)\right) $.

The analysis for $\begin{Bmatrix}n \\k \end{Bmatrix}_S$ roughly follows the one before. However, since the {generalized Stirling numbers of the second kind} count the number of partitions of $n$ into $k$ \textit{boxes} instead of $k$ \textit{cycles}, we apply inner $SET$ operators instead of $CYC$ operators.

When considering the EGF of $B_{n,S}$ we repeat the analysis for $\begin{bmatrix}n \\k \end{bmatrix}_S$, but now we apply an outer $SET$ operator instead of $SET_k$, since we care about an arbitrary number of cycles instead of just decompositions into $k$ cycles. 
\end{proof}

When we take $S=\{1,2,3,\ldots\}$, we  note that $\sum_{s \in S}\frac{z^{s}}{s!} = e^x-1$ and we recover the EGF for the classical Bell numbers $e^{e^x-1}$. In addition, taking $S=\{1,2,\ldots,m\}$ and $\{m,m+1,m+2, \ldots\}$ recover the generating functions $\exp\left(\sum_{i=1}^m \frac{z^{i}}{i!}\right)$,  $\exp\left(e^x-\sum_{i=0}^m \frac{z^{i}}{i!}\right)$, and so on \cite[Thm. 4.2, Thm. 5.2, (4.3)]{Moll1}.

The last two generating functions show that $B_{n,S}$ and $A_{n,S}$ are special cases of the \textit{complete Bell polynomials}, which are defined by the generating function $$ \exp\left(\sum_{n=0}^\infty a_n \frac{t^n}{n!}\right)=\sum_{n=0}^\infty B_n(a_0,\ldots,a_n)\frac{t^n}{n!}.$$ By comparing generating functions, we recover $B_{n,S}$ through the choice $$a_n = \begin{cases}
1, &n \in S; \\
0, &n \not\in S.
\end{cases}$$
Analogously, we recover $A_{n,S}$ through the choice $$a_n = \begin{cases}
(n-1)!, &n \in S; \\
0, &n \not\in S.
\end{cases}$$
%TODO connection to gould hopper

%From these explicit generating functions, we note that our construction closely resembles the construction of the Bernoulli-Barnes numbers from Bernoulli numbers. The construction of the Bernoulli-Barnes numbers is through their generating function $$\prod_{i=1}^n \frac{t}{e^{a_it}-1} =\sum_{n=0}^\infty B_n(\textbf{a})\frac{t^n}{n!}.$$
%It tensors together many copies of the generating function for the Bernoulli numbers $\frac{t}{e^t-1}$, which serve as the basic building block, while adding an arbitrary index set $\{a_i\}$.  When $\{a_i\}$ shrinks to a single element, we recover the Bernoulli numbers. In the current work, the generalized Bell and Stirling numbers are instead \textit{refinements} of the Bell and Stirling numbers. They begin with the basic building block $\sum_{s\in S}\frac{x^{s}}{s}$ or $e^{\frac{x^s}{s!}},$ while adding the index set $S$. 

\section{Composition sums}
Since all of the generating functions in Theorem \ref{main} can be easily represented as the composition of two functions, we can apply the \textbf{Fa\`a di Bruno formula} for higher order derivatives to derive new formulae for $A_{n,S}$, $B_{n,S}$, and the generalized Stirling numbers. 

Throughout this paper, we will consider \textbf{compositions} of an integer $n$. A composition $\pi$ is an \textit{ordered} tuple of positive integers called \textbf{parts} that add up to $n$ -- therefore, $(1,1,2)$, $(1,2,1)$, and $(2,1,1)$ all represent different compositions of $4$. We denote by $\mathcal{C}$ the set of all compositions, by $\mathcal{C}_n$ the set of all compositions of $n$, and by $|\thinspace\thinspace \cdot\thinspace \thinspace |$ the number of parts in a composition. For example, $|(1,1,2)|=3$. For the composition $\pi\in \mathcal{C}_n$, we denote the set of its parts by $\{\pi_i\}$. We also use the multi-index notation $g_\pi:=\prod_{\pi_i \in \pi}g_{\pi_i}$ and $\pi!=\prod_{\pi_i \in \pi}{\pi_i!}$, which enables us to state a convenient corollary of the classical {Fa\`a di Bruno formula}:

\begin{thm}\cite[Thm. 9]{Wakhare}\label{faa}
Let $g(z) = \sum_{n \geq 1} g_n z^n $ and $f(z) = \sum_{n\geq 0 }f_n z^n$. We then have the generating function identity
\begin{equation}
\label{eq:f(g(z))}
f(g(z)) = f_0 + \sum_{n \geq 1} z^n \sum_{\pi \in \mathcal{C}_n} f_{\vert \pi \vert} g_{\pi}.
\end{equation}
\end{thm}

\begin{thm}\label{compthm}
We have the following composition sum identities:
$$ \begin{bmatrix}n \\k \end{bmatrix}_S \frac{1}{n!} = \sum_{\substack{   \pi \in \mathcal{C}_n  \\ \pi_i \in S \\ |\pi|=k }} \frac{1}{\vert \pi \vert! \prod_{\pi_i \in \pi}{\pi_i}},$$
$$ \begin{Bmatrix}n \\k \end{Bmatrix}_S \frac{1}{n!} = \sum_{\substack{   \pi \in \mathcal{C}_n  \\ \pi_i \in S  \\ |\pi|=k}} \frac{1}{\vert \pi \vert! \pi!},$$
$$\frac{A_{n,S}}{n!} = \sum_{\substack{   \pi \in \mathcal{C}_n  \\ \pi_i \in S  }} \frac{1}{\vert \pi \vert! \prod_{\pi_i \in \pi}{\pi_i}},$$
$$\frac{B_{n,S}}{n!} = \sum_{\substack{   \pi \in \mathcal{C}_n  \\ \pi_i \in S  }} \frac{1}{\vert \pi \vert! \pi!}.$$
\end{thm}
\begin{proof}
To prove the formula for $A_{n,S}$, apply Theorem \ref{faa} with $f(z)=e^z$ and $g(z) =  \sum_{s \in S} \frac{z^{s}}{s}$. Therefore, 
$$g_{\pi} =  \begin{cases}\frac{1}{\prod_{\pi_i \in \pi}{\pi_i}}, & \text{every part of $\pi$ is $\in S$} \\ 0, & \text{else}  \end{cases}.$$ 
This allows us to rewrite our composition sum as follows: 
$$ \sum_{\pi \in \mathcal{C}_n} f_{\vert \pi \vert} g_{\pi} = \sum_{\pi \in \mathcal{C}_n} \frac{1}{\vert \pi \vert! \prod_{\pi_i \in \pi}{\pi_i}} \mathbbm{1}_{\pi_i\in S} = \sum_{\substack{   \pi \in \mathcal{C}_n  \\ \pi_i \in S  }} \frac{1}{\vert \pi \vert! \prod_{\pi_i \in \pi}{\pi_i}},$$ 
where the restriction on the parts of $S$ is transferred to the summation. Comparing coefficients completes the proof. The formula for $B_{n,S}$ is proven identically with $g(z) =  \sum_{s \in S} \frac{z^{s}}{s!}$ instead.

To prove the formula for $\begin{bmatrix}n \\k \end{bmatrix}_S$, apply Theorem \ref{faa} with $f(z)=\frac{z^k}{k!}$ and $g(z) =  \sum_{s \in S} \frac{z^{s}}{s}$. Now we have the piecewise definition $$f_{\vert \pi\vert} = \begin{cases}\frac{1}{\vert\pi\vert!}, & \vert\pi\vert=k \\ 0, & \text{else}  \end{cases}.$$
We can then simplify the composition sum since
$$ \sum_{\pi \in \mathcal{C}_n} f_{\vert \pi \vert} g_{\pi} =\sum_{\substack{   \pi \in \mathcal{C}_n  \\ \pi_i \in S  }} \frac{1}{\vert \pi \vert! \prod_{\pi_i \in \pi}{\pi_i}} \mathbbm{1}_{\vert\pi\vert=k}= \sum_{\substack{   \pi \in \mathcal{C}_n  \\ \pi_i \in S \\ |\pi|=k }} \frac{1}{\vert \pi \vert! \prod_{\pi_i \in \pi}{\pi_i}}.$$ 
The analysis for $\begin{Bmatrix}n \\k \end{Bmatrix}_S$ proceeds identically, but with  $g(z) =  \sum_{s \in S} \frac{z^{s}}{s!}$ instead.
\end{proof}
In the case $S=\mathbb{Z}_{\geq 1}$, we recover known expressions for the classic Stirling and Bell numbers. We can simply leave the restriction $\pi_i\in S$ off the summation, since we're now allowing all positive parts.

An avenue for further research is to find a combinatorial proof of any of the above identities. The appearance of a composition sum is not coincidental; it has a very natural connection to all of these numbers. Take the generalized Bell numbers as an example: we care about how we can partition $n$ elements into blocks, where each block has a size found in $S$. \textit{This is precisely an unordered composition of $n$ with each part in $S$}. When passing from compositions to partitions we are overcounting in some way which is accounted for by the composition sum.

\begin{cor}We have the inequalities
$$\begin{bmatrix}n \\k \end{bmatrix}_S \geq  \begin{Bmatrix}n \\k \end{Bmatrix}_S$$ and $$A_{n,S}\geq B_{n,S}.$$
\end{cor}
\begin{proof}
For $\pi\in\mathcal{C}_n$, we have $\pi! \geq \prod_{\pi_i\in\pi}\pi_i$ with equality if and only if $\pi_i=1$ -- that is, $\pi=(1,1,\ldots,1)$. 
Looking at the composition sum expressions in Theorem \ref{compthm}, this inequality will therefore be satisfied termwise. 
\end{proof}
%with equality iff $n=0$ (?) : when are they equal (depends on what values we're defining to be $0$)

%%%%%%%%%%%%%%%%%%%%%%%%%%%%%%%%%%%%%%%%%

\section{Recurrences}
We can then derive several general recurrences satisfied by these numbers, the first two of which generalize \cite[Thm. 4.2, Thm. 5.2]{Moll1}. All of the following identities also have combinatorial and umbral proofs; however, the generating function approach appears to be the most concise for the following simple identities. We also note that since all four types considered in this work are specializations of the Bell polynomials, several of the recurrences described here are specializations of recurrences for the Bell polynomials.
\begin{thm}\label{thm2}
We have the recurrences
$$ A_{n+1,S} = \sum_{s\in S}  \frac{n!}{(n-s+1)!} A_{n-s+1,S},$$
$$ B_{n+1,S} = \sum_{s\in S}  \binom{n}{s-1}B_{n-s+1,S},$$
$$ \begin{bmatrix}n+1 \\k \end{bmatrix}_S = \sum_{s\in S}\frac{n!}{(n-j+1)!}\begin{bmatrix}n-s+1 \\k-1 \end{bmatrix}_S,$$
$$ \begin{Bmatrix}n+1 \\k \end{Bmatrix}_S = \sum_{s\in S}\binom{n}{s-1}\begin{Bmatrix}n-s+1 \\k-1 \end{Bmatrix}_S.$$
\end{thm}
\begin{proof}
For concreteness we work only with the generalized Bell numbers. Taking a derivative of the generating function from Theorem \ref{main}, $\sum_{n=0}^\infty B_{n,S}\frac{z^n}{n!} = \exp \left(\sum_{s \in S} \frac{z^{s}}{s!} \right),$ gives
\begin{equation*}
\sum_{n=0}^\infty B_{n+1,S} \frac{z^n}{n!} = \left(  \sum_{s \in S} \frac{z^{s-1}}{(s-1)!} \right) \exp \left( \sum_{s \in S} \frac{z^{s}}{s!}\right) = \left(  \sum_{s \in S} \frac{z^{s-1}}{(s-1)!} \right)   \left(\sum_{n=0}^\infty B_{n,S} \frac{z^n}{n!}\right).
\end{equation*}
Taking a Cauchy product and comparing coefficients completes the proof.  The other proofs follow exactly the same lines, and follow from taking derivatives and then comparing coefficients.
\end{proof}
We also present a combinatorial proof of the second recurrence; the rest follow from similar reasoning. Start with $n+1$ elements and distinguish the first one. We can partition these $n+1$ elements in $B_{n+1,S}$ ways. However, we can also add $s-1$ elements to the first one, for any cardinality index $s\in S$. There are $\binom{n}{s-1}$ ways to pick these $s-1$ elements, and then $B_{n-(s-1),S}$ ways to partition the rest. Summing over $s$ completes the proof.

We can also derive \textit{lacunary recurrences} which take far fewer terms to compute. In particular, the formulae below have $\sigma$ completely free so they can take arbitrarily few terms to compute $B_{n,S \cup \sigma}$ if $\sigma$ is large enough relative to $n$. We also note that these recover \cite[Thm 4.1, Thm. 4.6]{Moll1} in the case $S=\{1,2,\ldots,m-1\}$ (or $S=\{1,2,\ldots,m\}$ for Theorem 4.6) and $\sigma=m$, where it was derived combinatorially.
\begin{thm}
If $\sigma\not\in S$, we have the recurrence $$B_{n,S\cup \sigma}=\sum_{i=0}^{\left \lfloor{\frac{n}{\sigma}}\right \rfloor } \frac{n!}{i!(n-\sigma i)!(\sigma!)^i} B_{n-\sigma i,S}. $$
If $\sigma \in S$, we have $$B_{n,S\setminus\sigma}=\sum_{i=0}^{\left \lfloor{\frac{n}{\sigma}}\right \rfloor } \frac{n! (-1)^i}{i!(n-\sigma i)!(\sigma!)^i} B_{n-\sigma i,S}.  $$
\end{thm}
\begin{proof}
We begin with the generating function from Theorem \ref{main}, $\sum_{n=0}^\infty B_{n,S}\frac{z^n}{n!} = \exp \left(\sum_{s \in S} \frac{z^{s}}{s!} \right).$

Assuming that $\sigma\not\in S$, we then multiply both sides by $\exp\left(\frac{z^{\sigma}}{\sigma!}\right)$. On one side, this gives $$\exp\left(\frac{z^{\sigma}}{\sigma!}\right)\exp \left(\sum_{s \in S} \frac{z^{s}}{s!} \right) = \exp \left(\sum_{s \in S \cup \sigma} \frac{z^{s}}{s!} \right)  = \sum_{n=0}^\infty B_{n,S\cup \sigma}\frac{z^n}{n!} .$$ On the other side, this then gives $$ \exp\left(\frac{z^{\sigma}}{\sigma!}\right) \sum_{n=0}^\infty B_{n,S}\frac{z^n}{n!}   =     \sum_{n=0}^\infty \frac{z^{\sigma n}}{(\sigma!)^n n!} \times \sum_{n=0}^\infty B_{n,S}\frac{z^n}{n!}  =    \sum_{n=0}^\infty z^n \sum_{i=0}^{\left \lfloor{\frac{n}{\sigma}}\right \rfloor } \frac{1}{i!(n-\sigma i)!(\sigma!)^i} B_{n-\sigma i,S}. $$ Comparing coefficients completes the proof. The second formula follows in an identical manner, after multiplying by $\exp\left(-\frac{z^{\sigma}}{\sigma!}\right)$ instead.
\end{proof}

We can also prove a similar theorem about $A_{n,S}$ using exactly the same reasoning, special cases of which are \cite[Thm. 5.1, Thm. 5.5]{Moll1}.
\begin{thm}
If $\sigma\not\in S$, we have the recurrence $$A_{n,S\cup \sigma}=\sum_{i=0}^{\left \lfloor{\frac{n}{\sigma}}\right \rfloor } \frac{n!}{i!(n-\sigma i)!\sigma^i} A_{n-\sigma i,S}. $$
If $\sigma \in S$, $$A_{n,S\setminus\sigma}=\sum_{i=0}^{\left \lfloor{\frac{n}{\sigma}}\right \rfloor } \frac{n! (-1)^i}{i!(n-\sigma i)!\sigma^i} A_{n-\sigma i,S}.  $$
\end{thm}

We can also prove other properties based on splitting $S$, generalizing \cite[Thm. 4.7]{Moll1}. We note that $\binom{n}{n_1,n_2,\ldots}$ is the multinomial coefficient $\frac{n!}{n_1!n_2!\cdots}$, and $\mathcal{C}_n$ in the set of compositions of $n$, described in the next section. 
\begin{thm}
Let $S=\cup_i S_i$, where $S_i \cap S_j = \emptyset$ for $i\neq j$, so that the sets $\{S_i\}$ partition $S$. Then $$B_n = \sum_{\pi \in \mathcal{C}_n} \binom{n}{\pi_1,\pi_2,\ldots} B_{\pi_1,S_1}B_{\pi_2,S_2}\cdots,$$ and $$A_n = \sum_{\pi \in \mathcal{C}_n} \binom{n}{\pi_1,\pi_2,\ldots} A_{\pi_1,S_1}A_{\pi_2,S_2}\cdots.$$
\end{thm}
\begin{proof}
We begin with the generating function for $B_{n,S}$, split it into parts corresponding to each $S_i$, and then re-express this as a convolution. In particular,
$$\sum_{n=0}^\infty B_{n,S}\frac{z^n}{n!} = \exp \left(\sum_{s \in S} \frac{z^{s}}{s!} \right) =\prod_i \exp \left(\sum_{s \in S_i} \frac{z^{s}}{s!} \right) =\prod_i \sum_{n=0}^\infty B_{n,S_1}\frac{z^n}{n!}.$$ Computing coefficients of $z^n$ in the resulting product completes the proof.
\end{proof}

\begin{cor}\label{div}
Let $S=\cup_i S_i$, where $S_i \cap S_j = \emptyset$ for $i\neq j$ so that the sets $\{S_i\}$ partition $S$, and take $p$ a prime. Then
$$B_{p,S} \equiv \sum_i B_{p,S_i} \pmod{p},$$ and $$A_{p,S} \equiv \sum_i A_{p,S_i} \pmod{p}.$$
\end{cor}
\begin{proof}
If a composition $\pi$ has each part strictly smaller than $n$ then $\binom{p}{\pi_1,\pi_2,\ldots} \equiv 0 \pmod{p}$ since $\binom{p}{\pi_1,\pi_2,\ldots} = \frac{p!}{\pi_1!,\pi_2!,\ldots}$ has a power of $p$ in the numerator which is not canceled by one in the denominator. Every term in the summation then vanishes modulo $p$, save for those of the form $\pi=(0,0,\ldots, \pi_i=p, \ldots,0)$. Taking into account the initial values $A_{0,S_j}=B_{0,S_j}=1$ for those compositions completes the proof.
\end{proof}

%https://cs.uwaterloo.ca/journals/JIS/VOL15/Mezo/mezo14.pdf
%http://people.brandeis.edu/~gessel/homepage/slides/comb-cong.pdf

We now discuss \textit{Spivey's formula} \cite{Spivey}, $$B_{n+m}=\sum_{k=0}^n\sum_{j=0}^m j^{n-k} \begin{Bmatrix}m \\j \end{Bmatrix}\binom{n}{k}B_k.$$ The cases $n=0$ and $m=1$ are $$B_m=\sum_{j=0}^m \begin{Bmatrix}m \\j \end{Bmatrix}$$ and $$B_{n+1} = \sum_{k=0}^n \binom{n}{k}B_k,$$ which are the two most basic recurrences satisfied by the Bell numbers. These have the analogs $$B_{m,S} = \sum_{j=0}^m \begin{Bmatrix}m \\j \end{Bmatrix}_S$$ and (from Theorem \ref{thm2}) $$B_{n+1,S} = \sum_{s\in S}  \binom{n}{s-1}B_{n-s+1,S},$$ suggesting the existence of a Spivey-type formula for the generalized Bell numbers. However, this analog has been particularly difficult to find.

\section{Umbral approach}
The goal of \textit{umbral calculus} is to simplify the proof of many identities by turning manipulation of sequences into manipulations of moments; this effectively turns questions about subscripts into questions about superscripts. We can do this if, for a sequence $\{a_n\},$ we can find a measure $\mu$ such that
\[
a_n = \int x^n d\mu\left(x\right).
\]

Ideally, $\mu$ would be a probability measure, i.e. a positive measure with unit total integral (which implies $a_0=1$),  but it is important to note that umbral calculus does not require that $\mu$ is a probability measure. However, not every sequence of real numbers $\{a_n\}$ can be a (probabilistic) moment sequence. Basic probabilistic results such as $\text{Var} (\A) = \E \A^2 - (\E \A)^2 \geq 0$ impose strong constraints on the sequence $\{a_n\}.$ For example, we cannot have the moment sequence $\{a_1 = 1, a_2=0\}$ since this yields a negative variance. 

However, if we are given a sequence that can be written as the moments of a measure $\mu$, we can then rewrite identities on sequences as identities on the moments of their underlying random variables. For a concrete example, consider the expression 
\begin{equation}
\label{binom}
\sum_{k=0}^n \binom{n}{k}a_k a_{n-k} =  \sum_{k=0}^n \binom{n}{k} \E \A_1^k \E \A_2^{n-k} =  \E \sum_{k=0}^n \binom{n}{k}  \A_1^k \A_2^{n-k} = \E (\A_1+\A_2)^n,
\end{equation}
where $\A_1$ and $\A_2$ are two independent random variables. Now, based on the moments of the distribution $\A_1+\A_2$, we obtain a nontrivial identity for the sequence $\{a_n\}$. For example, take $\A_1 \sim \Gamma_{p_1}$ and $\A_2 \sim \Gamma_{p_2}$ where $\sim$ denotes equality of distributiions. We have let $\Gamma_{p}$ denote a gamma-distributed random variable with density
\[
\frac{1}{\Gamma\left(p\right)} e^{-x} x^{p-1},
\]
so that $\A_1 + \A_2 \sim \Gamma\left(p_1+p_2\right).$ Since the moments of a Gamma random variable are the Pochhammer symbols 
\[
\mathbb{E}\Gamma_{p}^{k} = \left(p\right)_{k}:= \prod_{i=0}^{k-1}(p+i) ,
\]
identity \eqref{binom} translates into

\begin{equation}
\label{Chu}
\sum_{k=0}^n \binom{n}{k} \left(p_1\right)_k \left(p_2\right)_{n-k}=  \left(p_1+p_2\right)_n,
\end{equation}
which is nothing but the Chu-Vandermonde identity.

As can be seen from this example, the keys ideas at play here are the linearity of the expectation and the fact that $\E \A_1\A_2 = \E\A_1 \times \E \A_2$ for independent random variables. This allows us to transform questions involving the sequence $\{a_n\}$, which \textit{a priori} has no structure, into questions about moments of a random variable, which behave like powers. For more background into umbral approaches, consult \cite{Gessel, Roman}.

After some annoying reverse engineering, we can explicitly construct a random variable which has the complete Bell polynomials as a moment sequence, showing that \textbf{the Bell polynomials behave umbrally}. They are in fact special cases of a \textit{Sheffer sequence}, which is one of the most general polynomial sequences which satisfies umbral relations. The construction is quite involved, but allows us to systematize the derivation of identities for the Bell polynomials through umbral methods. In the next section, we will specialize this analysis to obtain results for the generalized Bell and Stirling numbers. 

We first describe a \textit{stable random variable}, which is a random variable $X$ such that there exist constants $c_n$ and $d_n$ so that $$X_1+X_2+\ldots+X_n \sim c_n X+d_n, $$ where the $\{X_i\}$ are independent and identically distributed copies of $X$. It is a general result \cite[Chapter 1]{Nolan} that a variable is stable if and only if it is equal to $aZ+b$, where $Z$ has characteristic function
$$
\E \exp\left(iuZ  \right) = \begin{cases}
\exp\left(- \vert u\vert^\alpha \left( 1-i\beta  \thinspace\text{sign}(u)\tan\left(\frac12\alpha\pi\right) \right)\right) \thinspace\thinspace\thinspace & \alpha \neq 1 ,   \\
\exp\left(- \vert u\vert \left( 1+i\beta  \thinspace\text{sign}(u)\frac{2}{\pi}\log\left(u\right) \right)\right) \thinspace\thinspace\thinspace & \alpha = 1.    
\end{cases}
$$ This is a two parameter distribution under the restrictions $0< \alpha \leq 2$ and $-1\leq \beta \leq 1$.

We introduce the \textit{symmetric $\alpha$-stable distribution} $S_\alpha$, which is the case $\beta=0$. It is symmetric around the origin and has moment generating function $\mathbb{E} e^{u \mathcal{S}_\alpha} =e^{ u ^{{\alpha}}}.$ Note that we still have the very strict requirement $0<\alpha\leq 2$, which is why we cannot assign $\A_j\sim S_j$ in the results that follow.

\begin{thm}\label{thm5}
Let $\A_j=\mathcal{W}_j\mathcal{S}_{\frac{1}{j}}^{-\frac{1}{j}}$ be a product of independent random variables. Here, $$\text{Pr}\left\{\mathcal{W}_j = \exp\left(\frac{2\pi i l}{j}\right)\right\} = \frac{1}{j},$$ $0\leq l \leq j-1$, so $\mathcal{W}_j$ is a complex-valued random variable equiprobable on the $j$-th roots of unity. $\mathcal{S}_{1/j}$ is a  symmetric $\alpha$-stable distribution with characteristic parameter $\frac{1}{j}$.

Then the random variable $$\A \sim \sum_{j=0}^\infty a_j^\frac{1}{j}\left(j!\right)^{-\frac{1}{j}}\A_j,$$ where the $\{\A_j\}$ are independent random variables distributed as above, has moment generating function $$\E e^{t\A} =\sum_{j=0}^\infty B_j(a_0,\ldots,a_n)\frac{t^j}{j!}.$$
Hence, $$ \E \A^n = B_n(a_0,\ldots,a_n).$$
\end{thm}
\begin{proof}
The variable $\A_j$ has been explicitly described before, and has moments given by \cite{Vignat1, Vignat2}
$$\E \A_j^n =\begin{cases} 0 &n \not\equiv 0 \pmod j \\  \frac{\left(qj\right)!}{q!} &n=qj \end{cases}.$$
Therefore, it has moment generating function $$\E e^{t\A_j} = \sum_{n=0}^\infty \E \A_j^n \frac{t^n}{n!} = \sum_{n=0}^\infty \frac{t^{jn}}{\left(jn\right)!} \frac{{\left(jn\right)}!}{n!} = e^{t^j}.$$
Due to the independence of the $\{\A_j\}$, we then know that 
$$\E e^{t\A} =  \prod_{j=0}^\infty \E e^{a_j^{\frac{1}{j}}\left(j!\right)^{-\frac{1}{j}}t\A_j} = \prod_{j=0}^\infty e^{a_j \frac{t^j}{j!}  } = \exp\left(\sum_{j=0}^\infty a_j \frac{t^j}{j!}\right) =\sum_{j=0}^\infty B_j(a_0,\ldots,a_n)\frac{t^j}{j!},$$
which completes the proof.
\end{proof}
%TODO finite vector a for Bell polys?

We can then exploit this representation to prove umbral identities. Most of the recurrence relations proven earlier also have simple umbral proofs. We can also construct a conjugate variable for $\mathcal{A}_j$ -- a random variable $\tilde{\A}_j$ that is independent of $\A_j$ and satisfies
\[
\mathbb{E} \left(\A_j + \tilde{\A}_j \right)^n = \delta_{n0},
\]
which will be extremely useful in generating umbral identities. Here, $\delta_{ij}$ is the Kronecker delta function.

\begin{thm}
Let $\tilde{\A}_j \sim \exp \left( {\frac{\pi i}{j}} \right) \A_j;$ then $\tilde{\A}_j$ is conjugate to $\A$. Moreover, $$\tilde{\A} \sim \sum_{j=0}^\infty a_j^\frac{1}{j}(j!)^{-\frac{1}{j}}\tilde{\A}_j,$$ where the $\{\tilde{\A}_j\}$ are independent random variables, is conjugate to $\A$. Hence, $$ \E (\A+\tilde{\A})^n 
%=  \E (\A_j+\tilde{\A}_j)^n 
= \delta_{n0}.$$
\end{thm}
\begin{proof}
Two independent random variables $\mathcal{X}$ and $\mathcal{Y}$ are \textit{conjugate} if and only if $\E e^{t\mathcal{X}}e^{t\mathcal{Y}} =1$; we can show this by expanding this the expectation as a series. Therefore, $$\E e^{t\mathcal{X}}e^{t\mathcal{Y}} = \sum_{n=0}^\infty \frac{t^n}{n!}\E (\mathcal{X}+\mathcal{{Y}})^n=1,$$ and $E (\mathcal{X}+\mathcal{{Y}})^n =\delta_{n0}$.

$\mathcal{A}_j$ and $\tilde{\A}_j$ can be seen as conjugate after manually comparing moment generating functions: $\tilde{A}_j$ has moments $$\E \tilde{\A}_j^n =\begin{cases} 0 &n \not\equiv 0 \pmod j \\  \frac{qj!}{q!} (-1)^q &n=qj \end{cases},$$
so $$\E e^{t\tilde{\A}_j} = \sum_{n=0}^\infty \E \tilde{\A}_j^n \frac{t^n}{n!} = \sum_{n=0}^\infty \frac{t^{jn}}{jn!} \frac{{jn}!}{n!} (-1)^n= e^{-t^j}.$$

Then, $\E e^{t\A_j}e^{t\tilde{\A}_j} = e^{t^j}e^{-t^j}=1$ and we are done. Due to the independence of the $\{\tilde{\A}_j\}$, we then compute 
$$\E e^{t\tilde{\A}} =  \prod_{j=0}^\infty \E e^{a_j^{\frac{1}{j}}\left(j!\right)^{-\frac{1}{j}}t\tilde\A_j} = \prod_{j=0}^\infty e^{-a_j \frac{t^j}{j!}  } = \exp\left(-\sum_{j=0}^\infty a_j \frac{t^j}{j!}\right).$$ Using this expression, $\E e^{t\A}e^{t\tilde{\A}} =1$ and we are done.
\end{proof}

While an umbral approach simply requires us to verify that a sequence is a moment sequence, by explicitly constructing the relevant random variable we can obtain a slew of new expression for the Bell and Stirling numbers. By specializing Theorem \ref{thm5} to the Bell and factorial numbers, we have the following important corollaries:
\begin{cor}\label{corb}
If $\{\A_j\}$ is a set of independent random variables, as given in Theorem \ref{thm5}, then for $$\A \sim \sum_{s \in S} \left(s!\right)^{-\frac{1}{s}}\A_s,$$ we have $$\E \A^n = B_{n,S}. $$
\end{cor}
\begin{cor}\label{cora}
Let $$\A \sim \sum_{s \in S} s^{-\frac{1}{s}}\A_s.$$ Then $$\E \A^n = A_{n,S}. $$
\end{cor}

\section{Next steps}
These results suggest the study of other types of restricted Bell numbers. An example would be ``even" or ``odd" Bell numbers, which arise from taking $S_e:=\{2,4,6,\ldots\}$ and $S_o:=\{1,3,5,\ldots\}$. The even and odd Bell numbers that count the number of partitions of $n$ into blocks of even (respectively, odd) size. While all of the above theorems apply to the even and odd Bell numbers, we also note that we have results like $B_n = \sum_{k=0}^n \binom{n}{k} B_{k,S_e}B_{n-k,S_o}$. This means that further information about the even and odd Bell numbers could yield new information about the classical case. Another interesting approach would be to compute the valuations of these dissections with respect to small primes, and how that relates to the $p$-adic valuation of the Bell numbers. We already initiated work in this direction with Corollary \ref{div}.

In fact, work has already been done on the dissection of $B_n$, depending on whether it splits $n$ into an even or odd number of blocks - the so called `complementary Bell numbers'. This is the subject of \textit{Wilf's Conjecture}, which asks whether the two values are ever equal \cite{Valerio}. However, our methods are more suited towards dissecting $B_n$ into whether is splits $n$ into blocks of even or odd size.

In any case, our generating functions provide \textit{refinements} of their classical counterparts, which enables us to dissect the classical case in various ways. For instance, the rank of a partition naturally dissects partitions into residues classes modulo $5$. It provides a refinement of the classical partition counting function $p(n)$. Analogously, we could consider Bell and Stirling numbers consisting of box sizes modulo a prime $p$. When $p=2$ we are led to the even and odd Bell numbers previously discussed. However, we could consider the sets $S_p^{(1)} = \{1, p+1,2p+1,\ldots\}$, $S_p^{(2)} = \{2, p+2,p+2,\ldots\}$, and more generally $S_p^{(i)} = \{i, p+i,2p+i,\ldots\}$ for $1\leq i \leq p$. Then $\mathbb{Z}_{\geq 1} = \cup_{i=1}^{p}S_p^{(i)}$ and we recover the classical case as the multinomial sum
$$B_n = \sum_{\pi \in \mathcal{C}_n} \binom{n}{\pi_1,\pi_2,\ldots} B_{\pi_1,S_p^{(1)}}B_{\pi_2,S_p^{(2)}}\cdots.$$

Moving away from the Bell numbers, there are still many fundamental properties of the generalized Stirling numbers that need to be explored. The Stirling numbers satisfy many classical recurrences and orthogonality relations, and it is an interesting question whether the generalized Stirling numbers satisfy similar identities. For instance, if $1\in S$ then matrices related to the generalized Stirling numbers have well defined inverses with combinatorial interpretations \cite{Engbers}. In some cases, the answer is no, and it appears that refining the Stirling numbers to an arbitrary index set $S$ destroys some essential structural properties.
%TODO orthogonality

For instance, even in the special case $S=\{1,2,\ldots,m\}$ the restricted Stirling numbers only satisfy the three term recurrences $$\begin{Bmatrix}n+1 \\k \end{Bmatrix}_S = k \begin{Bmatrix}n \\k \end{Bmatrix}_S+\begin{Bmatrix}n \\k-1 \end{Bmatrix}_S - \binom{n}{m}\begin{Bmatrix}n-m \\k-1 \end{Bmatrix}_S$$
and
$$\begin{bmatrix}n+1 \\k \end{bmatrix}_S = n \begin{bmatrix}n \\k \end{bmatrix}_S+\begin{bmatrix}n \\k-1 \end{bmatrix}_S - \frac{n!}{(n-m)!}\begin{bmatrix}n-m \\k-1 \end{bmatrix}_S.$$ 
This appears to be the `best possible' recurrence, in the sense that it involves a minimal number of terms. We conjecture that this implies that the generalized Stirling numbers \textit{cannot be represented as symmetric polynomials}. Following \cite[Chapter 1]{Macdonald}, when considering the polynomials ring $\mathbb{Z}[x_1,x_2,\ldots,x_n]$, the elementary symmetric function $e_k(x_1,\ldots,x_n) = \sum_{i_1<\cdots<i_k}x_{i_1}x_{i_2}\cdots x_{i_k}$ is defined through the generating product $\prod_{i=1}^n(1+tx_i) = \sum_{k=0}^\infty e_k t^k$ and the complete symmetric function $h_k(x_1,\ldots,x_n) = \sum_{i_1\leq\cdots\leq i_k}x_{i_1}x_{i_2}\cdots x_{i_k}$ is defined through the generating product $\prod_{i=1}^n(1-tx_i)^{-1} = \sum_{k=0}^\infty h_k t^k$. These are uniquely characterized by their initial values and the two term recurrences $$e_{n-j}(x_1,\ldots,x_{n-1})=e_{n-j}(x_1,\ldots,x_{n-2})+x_{n-1}e_{n-j-1}(x_1,\ldots,x_{n-2}) $$ and
$$ h_{n-j}(x_1,\ldots,x_{j})=h_{n-j}(x_1,\ldots,x_{j-1})+x_j h_{n-j-1}(x_1,\ldots,x_{j}).$$ 
Since the generalized Stirling numbers do not satisfy similar two term recurrences, they are not evaluations of elementary symmetric polynomials. Other generalizations of the Stirling number are, however, symmetric polynomial evaluations -- \cite{Mongelli} uses this technique of matching recurrence relations to establish that the \textit{Jacobi-Stirling numbers} are in fact symmetric polynomial evaluations.

This is important because it means that there is no simple product factorization of the generating function for the Stirling numbers. More concretely, the classical formulae $\prod_{i=0}^{n-1}(x+i) = \sum_{k=0}^n \begin{bmatrix}n \\k \end{bmatrix}x^k $ and $x^n = \sum_{k=0}^n \begin{Bmatrix}n \\k \end{Bmatrix}(x)(x-1)\cdots (x-k+1) $, expressing changes of bases of $\mathbb{Z}[x]$, have no general analog.

Alternatively, Stirling numbers arise as coefficients of the \textit{normal form} expansion of the differential operator $x\frac{d}{dx}$: $$\left(x\frac{d}{dx}\right)^n = \sum_{k=0}^n \begin{Bmatrix}n \\k \end{Bmatrix}x^k \frac{d^k}{dx^k}.$$ The work \cite{Bastiak} provides a comprehensive overview of work surrounding normal forms of differential operators. It remains to be seen whether there exists a modified differential operator which contains $ \begin{Bmatrix}n \\k \end{Bmatrix}$ in its normal form expansion.

Another interesting connection is to classical polynomial families. Applying Taylor's formula (where $D^n f(x):=\frac{d^n}{dx^n} f(x) \vert_{x=0}$) to our generating function gives us $$ \frac{B_{n,S}}{n!} = D^n  \exp \left( \sum_{s \in S} \frac{z^{s}}{s!}\right).$$ This yields a connection to the Hermite polynomials, which are defined as proportional to $D^n e^{-x^2}$. When $S=\{1,2\}$ (which classically yields the involution numbers), we can relate special values of $A_{n,S}=B_{n,S}$ and the Hermite polynomials. More generally, Gould-Hopper polynomials are defined in terms of nested derivatives of the exponentials of polynomials. Special values of those could probably be linked to $B_{n,S}$ in the future.

In this paper, we have only established a few basic properties for the generalized Bell, Stirling, and factorial numbers. Their general definition allows us to unite the study of many previously studied combinatorial quantities, and will hopefully continue to yield new results about the classical Bell and Stirling numbers.

\section{Acknowledgements}
I'd like to thank Christophe Vignat for everything -- I literally wouldn't be a researcher without him. He also provided significant help with the probabilistic parts of this paper, and pointed out the connection to Gould-Hopper polynomials. I'd also like to thank the hospitality of the Tulane mathematics department, and Diego Villamizar, Ira Gessel, and Jos\'e Ram\'irez for their assorted compliments and criticisms. Last but not least, here's a shoutout to Matthew Chung -- happy birthday, your gift is that you get your name in a scientific paper!

\end{document}